\documentclass[12pt,a4paper,draft]{amsart}
\usepackage[pagewise]{lineno}
\usepackage{graphicx}
\usepackage{amscd}
\usepackage[mathscr]{eucal}
\usepackage{faktor} 
\usepackage{mathrsfs}

\usepackage{amsmath}

\usepackage[all]{xy}
\usepackage{txfonts}
\usepackage{bigints}
\usepackage{multirow}
\usepackage{mathtools}
\allowdisplaybreaks

\newcommand{\adef}{\begin{defin}}
\newcommand{\zdef}{\end{defin}}
\newtheorem{defin}{Definition}
\newtheorem{theorem}{Theorem}[section]
\newtheorem{lemma}[theorem]{Lemma}
\newtheorem{proposition}[theorem]{Proposition}
\newtheorem{corollary}[theorem]{Corollary}

\numberwithin{equation}{section}

\numberwithin{equation}{section}

\newcommand{\vertiii}[1]{{\left\vert\kern-0.25ex\left\vert\kern-0.25ex\left\vert #1
    \right\vert\kern-0.25ex\right\vert\kern-0.25ex\right\vert}}

\def\block(#1,#2)#3{\multicolumn{#2}{c}{\multirow{#1}{*}{$ #3 $}}}

\DeclarePairedDelimiterX{\inp}[2]{\langle}{\rangle}{#1, #2}

\theoremstyle{definition}
\newtheorem{definition}[theorem]{Definition}

\theoremstyle{remark}

\numberwithin{equation}{section}

\address{Departamento de Matem\'atica, Instituto de Ci\^encias Matem\'aticas e de Computa\c{c}\~ao, Universidade de S\~ao Paulo, Avenida Trabalhador S\~ao-carlense, 400 - Centro
CEP: 13566-590 - S\~ao Carlos - SP, Brazil}
\email{willhans@icmc.usp.br}

\begin{document}

\title{Complex interpolation of Orlicz sequence spaces and its higher order Rochberg spaces}

\author{Willian H. G. Corrêa}

\thanks{This work was financed in part by CNPq (National Council for Scientific and Technological Development - Brazil, grant 140413/2016-2); Coordenação de Aperfeiçoamento de Pessoal de Nível Superior – Brasil (CAPES) – PDSE program 88881.134107/2016-0; São Paulo Research Foundation (FAPESP), grants 2016/25574-8, 2018/03765-1, 2019/09205-0. The author was also a member of project MTM2016-76958-C2-1-P of DGICYT (Spain).}

\begin{abstract}
We show that if $(\ell_{\phi_0}, \ell_{\phi_1})$ is a couple of suitable Orlicz sequence spaces then the corresponding Rochberg derived spaces of all orders associated to the complex interpolation method are Fenchel-Orlicz spaces. In particular, the induced twisted sums have the $(C[0, 1], \mathbb{C})$-extension property.
\end{abstract}


\maketitle

\section{Introduction}

The present paper is a spiritual continuation of \cite{ACK-Fenchel}, where the authors deal with extensions of Orlicz sequence spaces. An extension of a Banach space $Y$ is a quasi-Banach space $X$ such that $Y$ is isomorphic to a subspace of $X$ and the respective quotient is also isomorphic to $Y$. Androulakis, Cazaku and Kalton showed how Fenchel-Orlicz spaces (a natural generalization of Orlicz spaces to higher dimensions) may be used to obtain extensions of Orlicz sequence spaces with nontrivial type. 

More generally, given Banach spaces $Y$ and $Z$, a twisted sum of $Z$ and $Y$ (the order is important) is a quasi-Banach space $X$ such that $Y$ is isomorphic to a subspace of $X$ and $X/Y$ is isomorphic to $Z$. We may represent that in terms of a short exact sequence
\[
\xymatrix{0 \rightarrow Y \rightarrow X \rightarrow Z \rightarrow 0}
\]

The most famous case is when $Y = Z =  \ell_2$. If $X \neq \ell_2$ then $X$ is called a \emph{twisted Hilbert space}. The first example of twisted Hilbert space was given by Enflo, Lindenstrauss and Pisier \cite{Enflo_Lindenstrauss_Pisier_1975}, followed some years later by the example of Kalton and Peck \cite{KaltonPeck}.

One can go the extra mile and twist twisted sums: in \cite{HigherOrderRochberg} Rochberg presented his derived spaces associated to the complex method of interpolation. Such derived spaces give us families of twisted sums. Take, for example, $Z_2^{(1)} = \ell_2$ and $Z_2^{(2)} = Z_2$, the Kalton-Peck space. In \cite{TwistedTwisted}, Cabello Sánchez, Castillo and Kalton used Rochberg's construction to obtain a family of Banach spaces $(Z_2^{(n)})_{n \geq 3}$ that fit nicely into short exact sequences
\[
\xymatrix{0 \rightarrow Z_2^{(m)} \rightarrow Z_2^{(m+n)} \rightarrow Z_2^{(n)} \rightarrow 0}
\]

Using those spaces, they showed that being a twisted Hilbert space is not a 3-space property, since $Z_2^{(4)}$ is a Banach space which contains an isomorphic copy of $Z_2$, the respective quotient is isomorphic to $Z_2$, but $Z_2^{(4)}$ is not a twisted Hilbert space. We call the spaces $Z_2^{(n)}$ \emph{higher order extensions} of $\ell_2$.

Our aim in this paper is to give a description and to study the properties of the higher order extensions of Orlicz sequence spaces obtained by complex interpolation. We prove that those spaces are Fenchel-Orlicz spaces through an adaptation of the arguments of Androulakis, Cazaku and Kalton in \cite{ACK-Fenchel}.

The structure of Fenchel-Orlicz space is behind the astonishing fact that if $T : \ell_2 \rightarrow C[0, 1]$ is any operator then $T$ admits an extension to the Kalton-Peck space $Z_2$. We are able to show that our higher order extensions share the same property (considering complex scalars).

The structure of the paper is as follows: Section \ref{sec:Background} contains background on Fenchel-Orlicz spaces, complex interpolation and the twisted sums it generates. In Sections \ref{sec:quasiYoung} and \ref{sec:derivedareFenchel} we show how to obtain quasi-Young functions from complex interpolation of Orlicz sequence spaces, and that the derived spaces induced by the interpolation process agree with the Fenchel-Orlicz spaces generated by those quasi-Young functions. In Section \ref{sec:extension} we show that our twisted sums satisfy the aforementioned property of extension of operators with image in $C([0,1], \mathbb{C})$, and in Section \ref{sec:finalremarks} we conclude with a remark on how the results in \cite{ACK-Fenchel} may be seen in the context of complex interpolation.

\section{Background}\label{sec:Background}

\subsection{Fenchel-Orlicz spaces}

The standard reference for this topic is the work of Turett \cite{Turett}. Usually Young functions are defined for real vector spaces, but we will need to consider their complex version here. A Young function $\phi : \mathbb{C}^n \rightarrow [0, \infty)$ is a convex function such that $\phi(0) = 0$, $\lim_{t \rightarrow \infty} \phi(tx) = \infty$ for every $x \in \mathbb{C}^n \setminus \{0\}$, and $\phi(e^{is} x) = \phi(x)$ for every $x \in \mathbb{C}^n$, $s \in \mathbb{R}$. We define the Fenchel-Orlicz space
\[
\ell_{\phi} = \{(x^k) \in (\mathbb{C}^n)^{\mathbb{N}} : \exists \rho > 0 \mbox{ for which } \sum\limits_{k=1}^{\infty} \phi \Big(\frac{x_1^k}{\rho}, \cdots, \frac{x_n^k}{\rho}\Big) < \infty\}
\]
endowed with the complete norm
\[
\|(x^k)\|_{\phi} = \inf\{\rho > 0 : \sum\limits_{k=1}^{\infty} \phi \Big(\frac{x_1^k}{\rho}, \cdots, \frac{x_n^k}{\rho}\Big) \leq 1\}
\]

When $n = 1$ we have the definition of an Orlicz space, so we call $\phi$ an Orlicz function. If $\phi(t) > 0$ for $t > 0$ we say that $\phi$ is nondegenerate. We are particularly interested in this case because $\phi|_{[0, \infty)}$ is strictly increasing and therefore has an inverse. When we write $\phi^{-1}$ we will always be referring to this inverse. We also have another important particularity of Orlicz functions: it is enough to define $\phi$ on $[0, \infty)$, since for $x \in \mathbb{C}$ we have $\varphi(x) = \varphi(\left|x\right|)$.

A Young function $\phi$ on $\mathbb{C}^n$ satisfies the $\Delta_2$ condition (or is said to be in the class $\Delta_2$) if there is a constant $M > 0$ such that $\phi(2x) \leq M \phi(x)$ for every $x \in \mathbb{C}^n$. The following lemma is an easy exercise.

\begin{lemma}\label{lem:phi-constants}
Let $\phi$ be an Orlicz function in the class $\Delta_2$.
\begin{enumerate}
    \item There is $c \geq 0$ such that for every $x, y \in \mathbb{C}$
        \[
            \phi(x + y) \leq c (\phi(x) + \phi(y))
        \]
        
    \noindent In particular, if $\phi$ is strictly increasing then
    \[
    \phi^{-1}(s) + \phi^{-1}(t) \leq \phi^{-1}(c(s + t))
    \]
    for every $s, t \in [0, \infty)$
        
    \item If $a > 0$ then there is $D_{a} > 0$ such that
        \[
            \phi(ax) \leq D_{a} \phi(x)
        \]
    for every $x \in \mathbb{C}$.
\end{enumerate} 
\end{lemma}

In the definition of twisted sum the middle space of the short exact sequence is a quasi-Banach space. Accordingly, we will need to work with quasi-Young functions: we say that $\phi : \mathbb{C}^n \rightarrow [0, \infty)$ is quasi-Young if $\phi(0) = 0$, $\lim_{t \rightarrow \infty} \phi(tx) = \infty$ for every $x \in \mathbb{C}^n \setminus \{0\}$, $\phi(e^{is} x) = \phi(x)$ for every $x \in \mathbb{C}^n$, $s \in \mathbb{R}$, and there is a constant $C > 0$ such that
\[
\phi(tx + (1 - t)y) \leq C (t \phi(x) + (1 - t) \phi(y))
\]
for every $x, y \in \mathbb{C}^n$, $t \in [0, 1]$. We say that $C$ is a quasi-convexity constant for $\phi$.

If we define $\ell_{\phi}$ and $\|\cdot\|_{\phi}$ as above then $\|\cdot\|$ is a quasinorm and $\ell_{\phi}$ is a quasi-Banach space. Actually, we may easily adapt the argument of \cite[Proposition 2.1]{ACK-Fenchel} to complex variables and show that $\|\cdot\|_{\phi}$ is equivalent to a norm.

\subsection{Quasilinear maps}

The standard way to build a twisted sum of $Z$ and $Y$ is through quasilinear maps \cite{KaltonPeck}. Let $W$ be a vector space containing $Y$. We say that $F : Z \rightarrow W$ is quasilinear from $Z$ into $Y$ if it is homogeneous and there is a constant $Q > 0$ such that
\[
\|F(z_1 + z_2) - F(z_1) - F(z_2)\|_Y \leq Q(\|z_1\|_Z + \|z_2\|_Z)
\]
for every $z_1, z_2 \in Z$. Let
\[
Y \oplus_F Z = \{(w, z) : z \in Z, w - Fz \in Y\} \subset W \times Y
\]
be endowed with the complete quasinorm $\|(w, z)\| = \|w - Fz\|_Y + \|z\|_Z$. Then $Y \oplus_F Z$ is a twisted sum of $Z$ and $Y$, since the map $y \mapsto (y, 0)$ from $Y$ into $Y \oplus_F Z$ defines an isometry and the respective quotient is isometric to $Z$. For example, the Kalton-Peck space is defined through a quasilinear map $\mathcal{K}_2 : \ell_2 \rightarrow \ell_{\infty}$ from $\ell_2$ into $\ell_2$ such that
\[
\mathcal{K}_2(x) = \sum x_n \log \frac{\left|x_n\right|}{\|x\|_2} e_n
\]
for every $x \in \ell_2$ of finite support.

\subsection{Complex interpolation}

If the standard way to build twisted sums is through quasilinear maps, complex interpolation is a established tool to build quasilinear maps. Let $V$ be a Hausdorff topological vector space, and let $\overline{X} = (X_0, X_1)$ be a couple of Banach spaces for which there are continuous injections $i_j : X_j \rightarrow V$, $j = 0, 1$. We call such a couple \emph{compatible}. We can always suppose that $V$ is a Banach space. Indeed, consider the sum space
\[
\Sigma(\overline{X}) = \{i_0(x_0) + i_1(x_1) : x_0 \in X_0, x_1 \in X_1\}
\]
with the complete norm.
\[
\|x\| = \inf\{\|x_0\|_{X_0} + \|x_1\|_{X_1} : x = i_0(x_0) + i_1(x_1)\}
\]
We may then replace $i_j$ by the inclusion map $X_j \subset \Sigma(\overline{X})$, $j = 0, 1$. 

Let $\mathbb{S} = \{z \in \mathbb{C} : 0 \leq Re(z) \leq 1\}$, and let $\mathcal{F}(\overline{X})$ be the space of all bounded continuous functions $f : \mathbb{S} \rightarrow \Sigma(\overline{X})$ which are analytic on $\mathbb{S}^{\mathrm{o}}$ and such that the functions $t \mapsto f(j + it)$ are continuous and bounded from $\mathbb{R}$ into $X_j$, $j = 0, 1$. The space $\mathcal{F}(\overline{X})$ is a Banach space with the norm
\[
\|f\| = \sup_{\substack{j = 0, 1 \\ t \in \mathbb{R}}} \|f(j + it)\|_{X_j}
\]

For $\theta \in (0, 1)$, let $X_{\theta} = \{f(\theta) : f \in \mathcal{F}(\overline{X})\}$ endowed with the quotient norm
\[
\|x\|_{\theta} = \inf\{\|f\| : f \in \mathcal{F}(\overline{X}), f(\theta) = x\}
\]

Then $X_{\theta}$ is an interpolation space with respect to $(X_0, X_1)$. The classical example is $(\ell_{\infty}, \ell_{1})_{\theta} = \ell_{p_{\theta}}$, where $\frac{1}{p_{\theta}} = \theta$. For more information on interpolation, see \cite{BerghLofstrom}.

\subsection{Extensions induced by complex interpolation}

Quite surprisingly, a construction of Rochberg \cite{HigherOrderRochberg} yields higher order extensions of the interpolation space. Here we present the Rochberg spaces from the point of view of quasilinear maps \cite{TwistedTwisted}. If $f$ is a function on some complex domain with values in a Banach space, we let $\hat{f}[j; z] = \frac{f^{(j)}(z)}{j!}$ be its $j$-th Taylor coefficient at $z$. Let $d^1 X_{\theta} = X_{\theta}$, and suppose the space $d^n X_{\theta}$ has already been defined. There is a homogeneous function $B_{\theta}^n : d^n X_{\theta} \rightarrow \mathcal{F}(\overline{X})$ such that
\begin{enumerate}
    \item $\widehat{B_{\theta}^n}[j; \theta] = x_j$ for every $x = (x_{n-1}, \cdots, x_0) \in d^n X_{\theta}$ and every $0 \leq j \leq n-1$;
    \item There is a constant $C_n > 0$ independent of $x \in d^n X_{\theta}$ such that $\|B_{\theta}^n(x)\| \leq C_n \|x\|$.
\end{enumerate}

Let $\Omega_{\theta}^n : d^n X_{\theta} \rightarrow \Sigma(\overline{X})$ be defined by $\Omega_{\theta}^n(x) = \widehat{B_{\theta}^n}[n; \theta]$. Then $\Omega_{\theta}^n$ is quasilinear from $d^n X_{\theta}$ into $X_{\theta}$ and we can define the derived space
\[
d^{n+1} X_{\theta} = X_{\theta} \oplus_{\Omega_{\theta}^n} d^n X_{\theta} = \{(w, x) : x \in d^n X_{\theta}, w - \Omega^n_{\theta}(x) \in X_{\theta}\} \subset \Sigma(\overline{X}) \times d^n X_{\theta}
\]
endowed with the quasinorm presented above. We notice that the spaces $d^n X_{\theta}$ are independent of the choice of maps $B_{\theta}^n$ satisfying (1) and (2), up to equivalence of quasinorms. Also, the quasinorm on $d^n X_{\theta}$ is always equivalent to a norm.

These spaces form higher order extensions of the space $X_{\theta}$. Again, the classical example comes from $(\ell_{\infty}, \ell_{1})$: the derived spaces at $\theta = \frac{1}{2}$ are the spaces $Z_2^{(n)}$ of the introduction. By higher order extensions we mean that we have short exact sequences
\[
\xymatrix{0 \rightarrow d^m X_{\theta} \rightarrow d^{m+n} X_{\theta} \rightarrow d^n X_{\theta} \rightarrow 0}
\]
where the inclusion map is $x \mapsto (x, 0)$ and the quotient map is $(x, y) \mapsto y$.

The following fact will be important: once we have $B_{\theta}^1$ it is possible to define $B_{\theta}^n$ inductively. Indeed, let $\varphi : \mathbb{S} \rightarrow \overline{\mathbb{D}}$ be a conformal map such that $\varphi(\theta) = 0$. Suppose $B_{\theta}^{n}(x)$ is defined for every $x \in d^{n} X_{\theta}$ and let $x = (x_{n}, \cdots, x_0) \in d^{n+1} X_{\theta}$. In particular, $x_{n} - \Omega_{\theta}^{n}(x_{n-1}, \cdots, x_0) \in X_{\theta}$. We can take
\[
B_{\theta}^{n+1}(x)(z) = B_{\theta}^{n}(x_{n-1}, \cdots, x_0)(z) + \frac{n! \varphi^{n-1}(z)}{k_{n}} B_{\theta}^{1}(x_{n} - \Omega_{\theta}^{n}(x_{n-1}, \cdots, x_0))(z)
\]
where $k_{n}$ is a numerical constant depending only on the derivative of $\varphi$ at $\theta$ which ensures that $\widehat{B_{\theta}^{n+1}}[n; \theta] = x_{n}$.

\section{Obtaining quasi-Young functions from complex interpolation}\label{sec:quasiYoung}

From now on, unless otherwise stated, we suppose our Orlicz functions satisfy the $\Delta_2$ condition. It is a classical result that if $\phi_0, \phi_1 : \mathbb{C} \rightarrow [0, \infty)$ are nondegenerate Orlicz functions then for every $\theta \in (0, 1)$ the function $\phi_{\theta} : [0, \infty) \rightarrow [0, \infty)$ given by $\phi_{\theta}^{-1} = (\phi_0^{-1})^{1 - \theta} (\phi_1^{-1})^{\theta}$ defines an Orlicz function, and we have
\[
(\ell_{\phi_0}, \ell_{\phi_1})_{\theta} = \ell_{\phi_{\theta}}
\]
with equivalence of norms (see \cite{Gustavsson1977}; see also \cite{OrliczVector, OrliczFamilies} for generalizations). Notice that everything works equally well if we let $\ell_{\phi_0} = \ell_{\infty}$ or $c_0$ and take $\phi_0^{-1} \equiv 1$.

For $x \in \ell_{\phi_{\theta}}$ we can take
\[
B^1_{\theta}(x)(z) = \|x\|_{\phi_{\theta}} \sum\limits_{n=1}^{\infty} \phi_0^{-1}\Big(\phi_{\theta}\Big(\frac{\left|x_n\right|}{\|x\|_{\phi_{\theta}}}\Big)\Big)^{1-z} \phi_1^{-1}\Big(\phi_{\theta}\Big(\frac{\left|x_n\right|}{\|x\|_{\phi_{\theta}}}\Big)\Big)^z sgn(x_n) e_n
\]
where $sgn(z) = \frac{z}{\left|z\right|}$ if $z \in \mathbb{C} \setminus \{0\}$ and $sgn(0) = 0$. Therefore
\[
\Omega_{\theta}^1(x) = \sum\limits_{n=1}^{\infty} \log \frac{\phi_0^{-1}\Big(\phi_{\theta}\Big(\frac{\left|x_n\right|}{\|x\|_{\phi_{\theta}}}\Big)\Big)}{\phi_1^{-1}\Big(\phi_{\theta}\Big(\frac{\left|x_n\right|}{\|x\|_{\phi_{\theta}}}\Big)\Big)} x_n e_n
\]

\medskip

This together with the definition of the maps $B_{\theta}^{n}$, will inspire our definition of Young functions. Since we must deal with one coordinate at a time we cannot use the norm of the full vector to build a Young function, so we will consider the coordinates of $B_{\theta}^1(x)$ when $\|x\|_{\ell_{\phi_{\theta}}} = 1$. From now on we fix $\theta \in (0, 1)$.

\begin{definition}
For $x \in \mathbb{C}$ let $g_x : \mathbb{S} \rightarrow \mathbb{C}$ be given by
\[
g_x(z) = \phi_0^{-1}(\phi_{\theta}(\left|x\right|))^{1-z} \phi_1^{-1}(\phi_{\theta}(\left|x\right|))^z sgn(x)
\]
Let $n \geq 2$ and suppose that $g_x$ has been defined for every $x \in \mathbb{C}^{n-1}$. Let $x = (x_{n-1}, \cdots, x_0) \in \mathbb{C}^n$ and define $g_x : \mathbb{S} \rightarrow \mathbb{C}$ by
\[
g_x(z) = g_{(x_{n-2}, ..., x_0)}(z) + \frac{\varphi^{n-1}(z)}{k_n} g_{n!(x_{n-1} - \hat{g}_{(x_{n-2}, ..., x_{0})}[n-1; \theta])}(z)
\]
where $k_n$ is from the definition of $B_{\theta}^n$.
\end{definition}

Notice that $\hat{g}_x[j; \theta] = x_j$ for $0 \leq j \leq n-1$. Now we define our quasi-Young functions:
\begin{definition}\label{def-phi-p-n}
Let $\phi_{\theta, 1} = \phi_\theta$. For $n \geq 2$, let $\phi_{\theta, n} : \mathbb{C}^{n} \rightarrow [0, \infty)$ be given by
\[
\phi_{\theta, n}(x_{n-1}, ..., x_{0}) = \phi_{\theta, n-1}(x_{n-2}, ..., x_0) + \phi_\theta(x_{n-1} - \hat{g}_{(x_{n-2}, ..., x_0)}[n-1; \theta])
\]
\end{definition}

Of course, we must prove that those are indeed quasi-Young functions. We will need the following lemma.

\begin{lemma}\label{lem:gestimate}
For every $n \geq 1$ there are constants $\alpha_n, \beta_n$ such that for every $x = (x_{n-1}, \cdots, x_0) \in \mathbb{C}^n$ we have
\[
\left|g_{x}(it)\right| \leq \alpha_n \phi_0^{-1}(\beta_n \phi_{\theta, n}(x)) 
\]
and
\[
\left|g_{x}(1+it)\right| \leq \alpha_n \phi_1^{-1}(\beta_n \phi_{\theta, n}(x)) 
\]
\end{lemma}
\begin{proof}
We prove it by induction. The case $n = 1$ is straightforward, so suppose the result is true for $n-1$ and let $c$ and $D_{n!}$ be the constants of Lemma \ref{lem:phi-constants} for $\phi_{\theta}$. We have:
\begin{eqnarray*}
\left|g_{x}(it)\right| & = & \left|g_{(x_{n-2}, ..., x_0)}(it) + \frac{\varphi^{n-1}(it)}{k_n} g_{(n!(x_{n-1} - \hat{g}_{(x_{n-2}, ..., x_{0})}[n-1; \theta])}(it)\right| \\
    & \leq & \alpha_{n-1} \phi_0^{-1}(\beta_{n-1} \phi_{\theta, n-1}(x_{n-2}, ..., x_0)) + \frac{\alpha_{n-1}}{\left|k_n\right|} \phi_0^{-1}(\beta_{n-1} \phi_{\theta}(n!(x_{n-1} - \hat{g}_{(x_{n-2}, ..., x_{0})}[n-1; \theta])))  \\
    & \leq & \alpha_{n-1} \phi_0^{-1}(\beta_{n-1} \phi_{\theta, n-1}(x_{n-2}, ..., x_0)) + \frac{\alpha_{n-1}}{\left|k_n\right|} \phi_0^{-1}(\beta_{n-1} D_{n!} \phi_{\theta}(x_{n-1} - \hat{g}_{(x_{n-2}, ..., x_{0})}[n-1; \theta])) \\
    & \leq & \alpha_{n-1} \max\{1, \left|k_n\right|^{-1}\} \phi_0^{-1}(c \beta_{n-1} D_{n!} \phi_{\theta, n}(x)) 
\end{eqnarray*}
The proof for $1 + it$ is similar.
\end{proof}

\begin{theorem}\label{thm:quasiconvexfunctions}
For every $n \geq 1$ the function $\phi_{\theta, n}$ is quasi-Young.
\end{theorem}
\begin{proof}
We will prove it by induction. The case $n = 1$ is clear, so suppose that $\phi_{\theta, n-1}$ is quasi-Young with quasi-convexity constant $C_{n-1}$. Notice that $\phi_{\theta, n}(e^{is} x) = \phi_{\theta, n}(x)$,
\[
\lim_{t \rightarrow \infty} \phi_{\theta, n}(tx) = \infty
\]
for every $x \neq 0$ in $\mathbb{C}^{n}$, $s \in \mathbb{R}$, and $\phi_{\theta. n}(0) = 0$. It remains to prove the quasi-convexity of $\phi_{\theta, n}$. It is easy to see that this reduces to estimating
\[
\phi_\theta\Big(\sum\limits_{j} (t_j \hat{g}_{x^j})[n-1; \theta] - \hat{g}_{\sum\limits_{j} t_j x^j}[n-1; \theta]\Big)
\]
for $t_j \in [0, 1]$, $t_0 + t_1 = 1$, and $x^j \in \mathbb{C}^{n-1}$, $j = 0, 1$. Let $h_1 : \mathbb{S} \rightarrow \mathbb{C}$ be given by
\[
h_1(z) = \sum (t_j g_{x^j}(z)) -g_{\sum t_j x^j}(z)
\]

Notice that $\hat{h_1}[j; \theta] = 0$, for $j = 0, ..., n-2$. This means that the function $h_2(z) = \frac{h_1(z)}{(z - \theta)^{n-1}}$ is bounded on $\mathbb{S}$, analytic on $\mathbb{S}^{\mathrm{o}}$, and what we want to estimate is precisely $\phi_{\theta}(h_2(\theta))$.

Let $d = d(\theta, \partial \mathbb{S})$. By Lemma \ref{lem:gestimate} we have:
\begin{eqnarray*}
\left|h_2(it)\right| & \leq & \frac{1}{d^{n-1}} \Big(\sum (t_j |g_{x^j}(it)|) + |g_{\sum t_j x^j}(it)|\Big) \\
& \leq & \frac{1}{d^{n-1}} \sum (t_j \alpha_{n-1} \phi_0^{-1}(\beta_{n-1} \phi_{\theta, n-1}(x^j))) +  \frac{\alpha_{n-1}}{d^{n-1}} \phi_0^{-1}\Big(\beta_{n-1} \phi_{\theta, n-1}\Big(\sum t_j x^j\Big)\Big) \\
& \leq & \frac{\alpha_{n-1}}{d^{n-1}}\Big[ \phi_0^{-1}\Big(\beta_{n-1} \sum t_j\phi_{\theta, n-1}(x^j)\Big) +  \phi_0^{-1}\Big(\beta_{n-1} C_{n-1} \sum t_j \phi_{\theta, n-1}(x^j)\Big)\Big] \\
& \leq & \frac{2 \alpha_{n-1}}{d^{n-1}} \phi_0^{-1}\Big(\beta_{n-1} C_{n-1} \sum t_j \phi_{\theta, n-1}(x^j)\Big)
\end{eqnarray*}

We get a similar estimate for $1 + it$, substituting $\phi_0$ by $\phi_1$. So, by the three-lines lemma,
\begin{eqnarray*}
\left|h_2(\theta)\right| & \leq & \frac{2 \alpha_{n-1}}{d^{n-1}} \phi_0^{-1}\Big(\beta_{n-1} C_{n-1} \sum t_j \phi_{\theta, n-1}(x^j)\Big)^{1-\theta} \phi_1^{-1}\Big(\beta_{n-1} C_{n-1} \sum t_j \phi_{\theta, n-1}(x^j)\Big)^{\theta} \\
    & = & \frac{2 \alpha_{n-1}}{d^{n-1}} \phi_{\theta}^{-1}\Big(\beta_{n-1} C_{n-1} \sum t_j \phi_{\theta, n-1}(x^j)\Big) 
\end{eqnarray*}

Applying $\phi_{\theta}$ and letting $D_{\frac{2 \alpha_{n-1}}{d^{n-1}}}$ be the constant of Lemma \ref{lem:phi-constants} for $\phi_{\theta}$, we have
\begin{eqnarray*}
\phi_{\theta}(h_2(\theta)) & \leq & \phi_{\theta}\Big(\frac{2 \alpha_{n-1}}{d^{n-1}} \phi_{\theta}^{-1}\Big(\beta_{n-1} C_{n-1} \sum t_j \phi_{\theta, n-1}(x^j)\Big)\Big) \\
& \leq & D_{\frac{2 \alpha_{n-1}}{d^{n-1}}}\beta_{n-1} C_{n-1} \sum t_j \phi_{\theta, n-1}(x^j)
\end{eqnarray*}

It follows that $\phi_{\theta, n}$ is quasi-convex.
\end{proof}

\noindent \textbf{Observation:} The technique of the previous proof may also be used to show that each $\phi_{\theta, n}$ satisfies the $\Delta_2$ condition. Indeed, one may check that it is enough to estimate
\[
\phi_{\theta}(2\hat{g}_{x}[n-1; \theta] - \hat{g}_{2x}[n-1; \theta]),
\]
for $x \in \mathbb{C}^{n-1}$, which may done by taking $h_1 = 2 g_{x} - g_{2x}$.

\section{Derived spaces are Fenchel-Orlicz spaces}\label{sec:derivedareFenchel}

Our goal now is to prove that if $\phi_0$ and $\phi_1$ are nondegenerate Orlicz functions satisfying the $\Delta_2$ condition (or if we allow $\ell_{\phi_0} = \ell_{\infty}, c_0$) and we let $\overline{X} = (\ell_{\phi_0}, \ell_{\phi_1})$ then the derived space $d^n X_{\theta}$ is isomorphic to $\ell_{\phi_{\theta, n}}$. We will use the following result:

\begin{proposition}[\cite{ACK-Fenchel}, Proposition 3.2]
Let $X$ be a sequence space which is complete under the quasinorms $\|\cdot\|_1$ and $\|\cdot\|_2$, and suppose that the coordinate functionals are continuous in each norm. Then $\|\cdot\|_1$ and $\|\cdot\|_2$ are equivalent.
\end{proposition}

\begin{lemma}\label{lem:h_is_in_F}
Let $n \geq 1$ and $(x_{n-1}, \cdots, x_0) \in \ell_{\phi_{\theta, n}}$. Then the function $h_{(x_{n-1}, ..., x_0)}$ defined by
\[
h_{(x_{n-1}, ..., x_0)}(z) = \sum g_{(x_{n-1}(k), ..., x_0(k))}(z) e_k
\]
is in $\mathcal{F}(\overline{X})$.
\end{lemma}
\begin{proof}
We will prove it by induction in $n$. The base case is $n = 1$, so we have $x_0 \in \ell_{\phi_{\theta}}$. We must divide in two cases now:

\noindent \emph{Case 1: $\phi_0$ and $\phi_1$ are nondegenerate}

For $n_1 \leq n_2$ define
\[
h_{x_0}^{n_1, n_2}(z) = \sum\limits_{k = n_1}^{n_2} g_{x_0(k)}(z) e_k = \sum\limits_{k = n_1}^{n_2} \phi_0^{-1}(\phi_{\theta}(\left|x_0(k)\right|))^{1-z} \phi_1^{-1}(\phi_{\theta}(\left|x_0(k)\right|))^z sgn(x_0(k)) e_k
\]

Then
\[
\|h_{x_0}^{n_1, n_2}\|_{\mathcal{F}(\overline{X})} = \max\{\Big\|\sum\limits_{k=n_1}^{n_2} \phi_0^{-1}(\phi_{\theta}(\left|x_0(k)\right|)) e_k\Big\|_{\ell_{\phi_0}}, \Big\|\sum\limits_{k=n_1}^{n_2} \phi_1^{-1}(\phi_{\theta}(\left|x_0(k)\right|))e_k\Big\|_{\ell_{\phi_1}}\}
\]

Notice that $\sum\limits_{k=1}^{\infty} \phi_0^{-1}(\phi_{\theta}(\left|x_0(k)\right|)) e_k \in \ell_{\phi_0}$ and $\sum\limits_{k=1}^{\infty} \phi_1^{-1}(\phi_{\theta}(\left|x_0(k)\right|)) e_k \in \ell_{\phi_1}$, so that $\lim_{n_1, n_2} \|h_{x_0}^{n_1, n_2}\|_{\mathcal{F}(\overline{X})} = 0$, and therefore $h_{x_0}$ is the limit of the Cauchy sequence $(h_{x_0}^{1, n})_n$ in $\mathcal{F}(\overline{X})$.

\noindent\emph{Case 2: $\phi_0$ is degenerate}

In this case $\ell_{\phi_0} = \ell_{\infty} = \Sigma(\overline{X})$ with equivalence of norms. We have
\[
h_{x_0}(z) = \sum\limits_{k = 1}^{\infty} \phi_1^{-1}(\phi_{\theta}(\left|x_0(k)\right|))^z sgn(x_0(k)) e_k \in \ell_{\infty}
\]
If $y \in \ell_1$ then $(y, h(z)) = \sum\limits_{k=1}^{\infty} y(k) \phi_1^{-1}(\phi_{\theta}(\left|x_0(k)\right|))^z sgn(x_0(k))$ is absolutely convergent and the convergence is uniform on $z$. Therefore $(y, h(z))$ is continuous and bounded on $\mathbb{S}$ and analytic on $\mathbb{S}^{\mathrm{o}}$. Since $y \in \ell_1$ was arbitrary, this implies that $h_{x_0}$ is continuous and bounded on $\mathbb{S}$ and analytic on $\mathbb{S}^{\mathrm{o}}$ as a function with values in $\ell_{\infty}$. In particular, $t \mapsto h_{x_0}(it) \in \ell_{\infty}$ is continuous and bounded, and the argument of the previous case may be used to show that so is $t \mapsto h_{x_0}(1 + it) \in \ell_{\phi_1}$. This shows that $h_{x_0} \in \mathcal{F}(\overline{X})$.

Now, by induction, it is enough to prove that for $n \geq 2$ if $(x_{n-1}, \cdots, x_0) \in \ell_{\phi_{\theta, n}}$ then $(x_{n-1}(k) - \hat{g}_{(x_{n-2}(k), \cdots, x_0(k))}[n-1; \theta])_k \in \ell_{\phi_{\theta}}$. But this is a direct consequence of the definition of $\phi_{\theta, n}$.
\end{proof}

\begin{proposition}
If $n \geq 1$ then $d^n X_{\theta} = \ell_{\phi_{\theta. n}}$ as sets.
\end{proposition}
\begin{proof}
Again, the proof is by induction. The base case is the classical equality $\ell_{\phi_{\theta}} = (\ell_{\phi_0}, \ell_{\phi_1})_{\theta}$. Suppose the result is true for $n - 1$. We begin by proving the inclusion $\ell_{\phi_{\theta, n}} \subset d^n X_{\theta}$. Let $(x_{n-1}, ..., x_0) \in \ell_{\phi_{\theta, n}}$. This implies that $(x_{n-2}, ..., x_0) \in \ell_{\phi_{\theta, n-1}} = d^{n-1} X_{\theta}$. Let
\[
\Psi(x_{n-2}, ..., x_0) = (\hat{g}_{(x_{n-2}(k), ..., x_0(k))}[n-1; \theta])_k \in \mathbb{C}^{\mathbb{N}}
\]
Then $x_{n-1} - \Psi(x_{n-2}, ..., x_0) \in \ell_{\phi_{\theta}}$. We must show that $x_{n-1} - \Omega_{\theta}^{n-1}(x_{n-2}, ..., x_0) \in \ell_{\phi_{\theta}}$, so, it is enough to show that $\Omega_{\theta}^{n-1}(x_{n-2}, ..., x_0) - \Psi(x_{n-2}, ..., x_0) \in \ell_{\phi_{\theta}}$. 
To see that, let
\[
h_{(x_{n-2}, ..., x_0)}(z) = \sum g_{(x_{n-2}(k), ..., x_0(k))}(z) e_k
\]
Then $h \in \mathcal{F}(\overline{X})$ by Lemma \ref{lem:h_is_in_F} and the function $l = h_{(x_{n-2}, ..., x_0)} - B_{\theta}(x_{n-2}, ..., x_0)$ is such that $\hat{l}[j; \theta] = 0$, $j = 0, ..., n-2$. It follows that $\hat{l}[n-1; \theta] \in \ell_{\phi_{\theta}}$ and therefore $\Omega_{\theta}^{n-1}(x_{n-2}, ..., x_0) - \Psi(x_{n-2}, ..., x_0) \in \ell_{\phi_{\theta}}$.

To prove the reverse inclusion, if $(x_{n-1}, ..., x_0) \in d^n X_{\theta}$ then $(x_{n-2}, ..., x_0) \in d^{n-1} X_{\theta} = \ell_{\phi_{\theta, n-1}}$ and $x_{n-1} - \Omega_{\theta}^{n-1}(x_{n-2}, ..., x_0) \in \ell_{\phi_{\theta}}$. So, by the previous calculation, we have $x_{n-1} - \Psi(x_{n-2}, ..., x_0) \in \ell_{\phi_{\theta}}$, and therefore $(x_{n-1}, ..., x_0) \in \ell_{\phi_{\theta, n}}$.
\end{proof}

\begin{proposition}
If $n \geq 1$ then the coordinate functionals on $\ell_{\phi_{\theta, n}}$ are continuous.
\end{proposition}
\begin{proof}
One may check that $\phi_{\theta, n}$ is continuous, so if $\|(x_{n-1}, ..., x_0)\|_{\ell_{\phi_{\theta, n}}} \leq 1$ then
\[
\sum \phi_{\theta, n}(x_{n-1}(k), ..., x_0(k)) \leq 1
\]
In particular, $\phi_{\theta, n}(x_{n-1}(k), ..., x_0(k)) \leq 1$ for every $k$. By the continuity of $\phi_{\theta, n}$ the set $\phi_{\theta, n}^{-1}[0, 1]$ is bounded, which implies that the coordinate functionals are bounded.
\end{proof}

\begin{proposition}
If $n \geq 1$ then the coordinate functionals on $d^n X_{\theta}$ are continuous.
\end{proposition}
\begin{proof}
The result is true for $n = 1$, since $d^1 X_{\theta}$ is an Orlicz space. Suppose the result is true for $n-1$ and that $\|(x_{n-1}, ..., x_0)\|_{d^n X_{\theta}} \leq 1$. It follows that $\|(x_{n-2}, ..., x_0)\|_{d^{n-1} X_{\theta}} \leq 1$, and by induction hypothesis there is $M > 0$ such that $\left|x_j(k)\right| \leq M$, $k \in \mathbb{N}$, $j = 0, ..., n-2$. It remains to prove that $x \mapsto x_{n-1}(k)$ is bounded.

Recall that $\|B_{\theta}^{n}(x_{n-1}, ..., x_0)\| \leq C_n$. Now, for each $n \in \mathbb{N}$ the map $\delta_{\theta, n} : \mathcal{F}(\overline{X}) \rightarrow \Sigma(\overline{X})$ given by $\delta_{\theta, n}(f) = f^{(n)}(\theta)$ is continuous, so there is $M > 0$ independent of $x$ such that
\[
\|\Omega_{\theta}^n(x_{n-1}, ..., x_0)\|_{\Sigma(\overline{X})} \leq M
\]
It is clear that the coordinate functionals on $\Sigma(\overline{X})$ are continuous too, so there is $N > 0$ independent of $x$ (and of $k$) such that
\[
\left|\Omega_{\theta}^n(x_{n-1}, ..., x_0)(k)\right| \leq N
\]
Now the result follows from $\left|x_{n-1}(k)\right| \leq \left|x_{n-1}(k) - \Omega_{\theta}^n(x_{n-1}, ..., x_0)(k)\right| + \left|\Omega_{\theta}^n(x_{n-1}, ..., x_0)(k)\right|$.
\end{proof}

All this implies
\begin{theorem}
For each $n \geq 1$ the identity $d^n X_{\theta} = \ell_{\phi_{\theta, n}}$ is an isomorphism.
\end{theorem}

\section{The $C[0, 1]$-extension property}\label{sec:extension}

If $X$ is a real Banach space and $Y$ is a subspace of $X$, we say that the pair $(Y, X)$ has the $C[0, 1]$-extension property if every operator $T : Y \rightarrow C[0, 1]$ admits an extension to $X$. If $X$ is a complex Banach space we will accordingly deal with the $C([0, 1], \mathbb{C})$-extension property, which is defined analogously.

A combination of \cite[Theorem 4.1]{ACK-Fenchel} and the results of \cite{KaltonExtension} show that if $\phi$ is a real Young function in the class $\Delta_2$ then $(Y, \ell_{\phi})$ has the $C[0, 1]$-extension property for every subspace $Y$ of $\ell_{\phi}$.

The goal of this section is to prove the following result:

\begin{theorem}\label{thm:ext_property}
For every $n \leq m$ the pair $(\ell_{\phi_{\theta, n}}, \ell_{\phi_{\theta, m}})$ has the $C([0,1], \mathbb{C})$-extension property.
\end{theorem}

Recall that if $X$ is a real Banach space then on the complexification $X \oplus_{\mathbb{C}} X$ we put the norm
\[
\|(x, y)\| = \sup_{\theta \in [0, 2\pi]} \|\cos(\theta) x + \sin(\theta) y\|
\]

The following result is clear:

\begin{proposition}
Let $(Y, X)$ be a pair of real Banach spaces with the $C[0, 1]$-extension property. Then $(Y \oplus_{\mathbb{C}} Y, X \oplus_{\mathbb{C}} X)$ has the $C([0, 1], \mathbb{C})$-extension property.
\end{proposition}


For a complex Young function $\phi$ we will let $\ell_{\phi}(\mathbb{R})$ be the real sequence space $\ell_{\phi|_{\mathbb{R}}}$.

\begin{lemma}\label{lem:real_conformal}
For every $\theta \in (0, 1)$ there is a conformal map $\varphi : \mathbb{S} \rightarrow \mathbb{D}$ such that $\varphi(t) \in \mathbb{R}$ for every $t \in (0, 1)$ and $\varphi(\theta) = 0$. In particular, $\varphi^{(n)}(t) \in \mathbb{R}$ for every $n \geq 1$ and $t \in (0, 1)$, and therefore we may take $k_n$ real.
\end{lemma}
\begin{proof}
Consider the conformal map $\chi : \mathbb{S} \rightarrow \mathbb{D}$ given by
\[
\chi_{\theta}(z) = \frac{e^{i\pi z} - e^{i \pi \theta}}{e^{i\pi z} - e^{- i \pi \theta}}
\]
We have $\chi_{\theta}(\theta) = 0$. Also, one may check that if we write $\chi_{\theta}(z) = (f_1(z), f_2(z)) \in \mathbb{R}^2$ then the ratio $\frac{f_2(t)}{f_1(t)}$ is constant for $t \in (0, 1)$, which means that we may obtain $\varphi$ as in the enunciate by multiplying $\chi_{\theta}$ by a modulus one constant. The remark about $k_n$ follows from noticing that it is defined in terms of the derivative of $\varphi$ at $\theta$.
\end{proof}

\begin{lemma}
For every $n \geq 1$ there is a constant $a_n$ such that $\phi_{\theta, n}(x) \leq a_n \phi_{\theta, n}(x + iy)$ for every $x, y \in \mathbb{R}^n$.
\end{lemma}
\begin{proof}
For the base case we may take $a_1 = 1$. By induction, it is enough to show that there is a constant $c_n$ such that
\[
\phi_\theta(x_{n-1} - \hat{g}_{(x_{n-2}, ..., x_0)}[n-1; \theta]) \leq c_n \phi_{\theta, n}(x_{n-1} + i y_{n-1}, \cdots, x_0 + i y_0)
\]
for every $x, y \in \mathbb{R}^n$.

Let $c$ be the constant of Lemma \ref{lem:phi-constants} for $\phi_{\theta}$. We have that $\phi_\theta(x_{n-1} - \hat{g}_{(x_{n-2}, ..., x_0)}[n-1; \theta])$ is bounded by
\[
\phi_\theta(x_{n-1} - \hat{g}_{(x_{n-2}, ..., x_0)}[n-1; \theta] + i(y_{n-1} - Im \mbox{ } \hat{g}_{(x_{n-2} + i y_{n-2}, ..., x_0 + i y_0)}[n-1; \theta]))
\]
which in turn is smaller or equal to $c$ times
\begin{eqnarray*}
&& \phi_{\theta}(x_{n-1} + i y_{n-1} - \hat{g}_{(x_{n-2} + i y_{n-2}, ..., x_0 + i y_0)}[n-1; \theta]) \\
+ &&\phi_{\theta}(Re \mbox{ } \hat{g}_{(x_{n-2} + i y_{n-2}, ..., x_0 + i y_0)}[n-1; \theta] - \hat{g}_{(x_{n-2}, ..., x_0)}[n-1; \theta])
\end{eqnarray*}

Now it is enough to bound $\phi_{\theta}(Re \mbox{ } \hat{g}_{(x_{n-2} + i y_{n-2}, ..., x_0 + i y_0)}[n-1; \theta] - \hat{g}_{(x_{n-2}, ..., x_0)}[n-1; \theta])$ in terms of $\phi_{\theta, n-1}(x_{n-2} + i y_{n-2}, \cdots, x_0 + i y_0)$. So let $h_1 = g_{x_{n-2} + iy_{n-2}, \cdots, x_0 + i y_0} - g_{x_{n-2}, \cdots, x_0} - i g_{y_{n-2}, \cdots, y_0}$ and proceed as in the proof of Theorem \ref{thm:quasiconvexfunctions}.
\end{proof}

Now, Theorem \ref{thm:ext_property} is simply a consequence of the following:

\begin{corollary}
For every $n \geq 1$ there is a constant $a_n$ such that $\|x\|_{\ell_{\phi_{\theta ,n}}} \leq a_n \|x + iy\|_{\ell_{\phi_{\theta ,n}}}$ for every $x, y \in (\mathbb{R}^n)^{\mathbb{N}}$. In particular, $\ell_{\phi_{\theta, n}}$ is $\mathbb{R}$-isomorphic to the complexification $\ell_{\phi_{\theta, n}}(\mathbb{R}) \oplus_{\mathbb{C}} \ell_{\phi_{\theta, n}}(\mathbb{R})$.
\end{corollary}


From our results, the Johnson-Szankowski twisted Hilbert spaces $Z(\mathcal{JS})$ of \cite{SomeMoreTwisted} are Fenchel-Orlicz spaces up to a renorming, and $(\ell_2, Z(\mathcal{JS}))$ has the $C([0, 1], \mathbb{C})$-extension property. The Johnson-Szankowski twisted Hilbert spaces are examples of HAPpy spaces which are not asymptotically Hilbertian.

\section{Final remark}\label{sec:finalremarks}

In \cite{ACK-Fenchel}, given an Orlicz sequence space $\ell_{\phi}$ with nontrivial type, the authors use Lipschitz functions to build extensions of $\ell_{\phi}$. If we take the identity as Lipschitz function and $\ell_{\phi}$ is $p$-convex and $q$-concave for nontrivial $p$ and $q$, one may check that the extension of $\ell_{\phi}$ obtained in \cite{ACK-Fenchel} corresponds to the one induced by the couple $(\ell_{\infty}, (\ell_{\phi})_{(p)})$ at $\frac{1}{p}$, where $X_{(p)}$ is the $p$-concavification of $X$. It follows that we automatically get higher order extensions of $\ell_{\phi}$ in that case. It would be interesting to obtain higher order extensions when using other functions.

\bibliographystyle{amsplain}
\bibliography{refs}

\end{document}